\documentclass[11pt,letterpaper,reqno]{amsart}
\usepackage{graphicx, psfrag, amscd, amssymb, young}
\usepackage{mathrsfs}

\newtheorem{lem}{Lemma}[section]
\newtheorem{thm}[lem]{Theorem}
\newtheorem{pro}[lem]{Proposition}
\newtheorem{cor}[lem]{Corollary}
\newtheorem{exa}[lem]{Example}
\newtheorem{con}[lem]{Conjecture}

\newcommand{\B}{{\mathscr{B}}}

\newcommand{\I}{{\mathscr{I}}}

\newcommand{\N}{{\textsf{N}}}
\newcommand{\E}{{\textsf{E}}}

\newcommand{\Des}{{\textsf{Des}}}
\newcommand{\des}{{\textsf{des}}}

\newcommand{\inv}{{\textsf{inv}}}

\newcommand{\lead}{{\textsf{lead}}}
\newcommand{\stat}{{\textsf{stat}}}
\newcommand{\ldes}{{\textsf{ldes}}}
\newcommand{\maj}{{\textsf{maj}}}

\newcommand{\sump}{{\textsf{sump}}}

\renewcommand{\P}{{\mathscr{P}}}

\newcommand{\F}{{\mathscr{F}}}

\newcommand{\X}{{\mathscr{X}}}

\title[Folding Phenomenon of Parity-balance of Major-balance Identities]{Folding Phenomenon of Major-balance Identities on Restricted Involutions}

\author{Tung-Shan Fu}
\address{General Education Center, National Pingtung University, Pingtung 900, Taiwan, ROC}
\email{tsfu@mail.nptu.edu.tw}

\author{Hsian-Chun Hsu}
\address{Department of Mathematics, National Taiwan Normal University, Taipei 106, Taiwan, ROC}
\email{hchsu0222@gmail.com}

\author{Hsin-Chieh Liao}
\address{Department of Mathematics, National Taiwan Normal University, Taipei 106, Taiwan, ROC}
\email{jeffliao1@gmail.com}

\begin{document}

\begin{abstract} In this paper we prove a
refined major-balance identity on the $321$-avoiding involutions of length $n$, respecting the leading element of permutations. The proof is based on a sign-reversing involution on the lattice paths within a $\lfloor\frac{n}{2}\rfloor\times\lceil\frac{n}{2}\rceil$ rectangle. Moreover, we prove affirmatively a question about refined major-balance identity on the $123$-avoiding involutions, respecting the number of descents. 
\end{abstract}

\maketitle

\section{Introduction}
Let $\mathfrak{S}_n$ be the set of permutation of $\{1,2,\dots,n\}$. A permutation $\sigma=\sigma_1\cdots\sigma_n\in\mathfrak{S}_n$ is called 321-\emph{avoiding} (123-\emph{avoiding}, respectively) if it has no decreasing (increasing, respectively) subsequence of length three. Let $\mathfrak{S}_n(321)$ ($\mathfrak{S}_n(123)$, respectively) be the set of 321-avoiding (123-avoiding, respectively) permutations in $\mathfrak{S}_n$. It is known that $|\mathfrak{S}_n(321)|=|\mathfrak{S}_n(123)|=C_n=\frac{1}{n+1}{{2n}\choose {n}}$, the $n$th Catalan number. 

\subsection{Sign-balance for restricted permutations}
The sign-balance of restricted permutations is an interesting theme in enumerative combinatorics. Simion and Schmidt  \cite{Simion-Schmidt} proved the following sign-balance property of $\mathfrak{S}_n(321)$.
\[
\sum_{\sigma\in\mathfrak{S}_n(321)} (-1)^{\inv(\sigma)} =\left\{
                \begin{array}{ll}
                C_{\frac{n-1}{2}} & \mbox{if $n$ is odd} \\
                0                 & \mbox{if $n$ is even,}
                \end{array}
         \right.
\]
where $\inv(\sigma)=\#\{(\sigma_i,\sigma_j):\sigma_i>\sigma_j$ and $i<j\}$ is the \emph{inversion number} of $\sigma$.
Adin and Roichman \cite{AdinRoichman} proved a refinement of this result, respecting the position of the last descent ($\ldes$) of $\sigma$, i.e., $\ldes(\sigma)=\mathrm{max}\{i:\sigma_i>\sigma_{i+1}$ and $1\le i\le n-1\}$.

\begin{thm}[{\bf Adin-Roichman}] For all $n\ge 1$, the following identities hold.
\begin{enumerate}
\item ${\displaystyle
\sum_{\sigma\in\mathfrak{S}_{2n+1}(321)} (-1)^{\inv(\sigma)}q^{\ldes(\sigma)} = \sum_{\sigma\in\mathfrak{S}_n(321)} q^{2\cdot\ldes(\sigma)} 
}$
\item ${\displaystyle
\sum_{\sigma\in\mathfrak{S}_{2n}(321)} (-1)^{\inv(\sigma)}q^{\ldes(\sigma)} = (1-q)\sum_{\sigma\in\mathfrak{S}_n(321)} q^{2\cdot\ldes(\sigma)}.
}$
\end{enumerate}
\end{thm}

Later, Reifegerste \cite{Reifegerste} proved an analogous refinement for another permutation statistic, the
length of the longest increasing subsequence of the permutations. Eu et al.  \cite{EFPT-Alt} turned to other families of restricted permutations and obtained refined sign-balance result for 321-avoiding alternating permutations,
respecting the leading element and the last element of  permutations, respectively.

These results share a folding phenomenon that with respect to a certain statistic the sign-balance generating function for restricted permutations of length $2n$ essentially equals the ordinary generating function for the permutations of length $n$.
Eu et al. \cite{EFPT-In321} described the folding phenomenon in the framework
\[
\sum_{\pi\in\X_{2n}\mbox{ {\tiny  or} }\X_{2n+1}} (-1)^{\stat_1(\pi)}q^{\stat_2(\pi)}=f(q)\sum_{\pi\in\X_n} q^{2\cdot\stat_2(\pi)},
\]
where $\X_n$ is a family of combinatorial objects of size $n$ with statistics $\stat_1$ and $\stat_2$, and $f(q)$ is a rational function.
In this paper, we present an instance of such a phenomenon on 321-avoiding involutions.

\subsection{Major-balance for 321-avoiding involutions}
The {\em descent set} of $\sigma$ is defined as $\Des(\sigma)=\{i:\sigma_i>\sigma_{i+1}, 1\le i\le n-1\}$,
and the {\em descent number} ($\des$) and {\em major index} ($\maj$) of $\sigma$ are defined by
\[
\des(\sigma)= \sum_{i\in\Des(\sigma)} 1\quad\mbox{and}\quad\maj(\sigma) = \sum_{i\in\Des(\sigma)} i.
\]
Recall that a permutation $\sigma$ is called an \emph{involution} if and only if $\sigma^{-1}=\sigma$. Let $\I_n(321)$ ($\I_n(123)$, respectively) be the set of involutions in $\mathfrak{S}_n(321)$ ($\mathfrak{S}_n(123)$, respectively).
Simion and Schmidt \cite{Simion-Schmidt} proved that
$|\I_n(321)|=|\I_n(123)|={{n}\choose{n/2}}$.
Recently, Eu et al. \cite{EFPT-In321} proved the following refined major-balance result on 321-avoiding involutions.

\begin{thm}[{\bf Eu-Fu-Pan-Ting}] \label{thm:ARM-identity-for-InA321} For all $n\ge 1$, the following identities hold.
\begin{enumerate}
\item ${\displaystyle
\sum_{\sigma\in\I_{4n}(321)} (-1)^{\maj(\sigma)}q^{\des(\sigma)} =\sum_{\sigma\in\I_{2n}(321)}q^{2\cdot\des(\sigma)},
}$
\item ${\displaystyle
\sum_{\sigma\in\I_{4n+2}(321)} (-1)^{\maj(\sigma)}q^{\des(\sigma)} =(1-q)\sum_{\sigma\in\I_{2n}(321)}q^{2\cdot\des(\sigma)},
}$
\item ${\displaystyle
\sum_{\sigma\in\I_{2n+1}(321)} (-1)^{\maj(\sigma)}q^{\des(\sigma)}=\sum_{\sigma\in\I_{n}(321)}q^{2\cdot\des(\sigma)}.
}$
\end{enumerate}
\end{thm}
Meanwhile, they asked a question about an analogous result for 123-avoiding involutions.

\begin{con}[{\bf Eu-Fu-Pan-Ting}] \label{con:max-des-identity-for-InA123} For all $n\ge 1$, the following identities hold.
\begin{enumerate}
\item ${\displaystyle
\sum_{\sigma\in\I_{4n}(123)} (-1)^{\maj(\sigma)}q^{\des(\sigma)} =q\sum_{\sigma\in\I_{2n}(123)}q^{2\cdot\des(\sigma)},
}$
\item ${\displaystyle
\sum_{\sigma\in\I_{4n+2}(123)} (-1)^{\maj(\sigma)}q^{\des(\sigma)} =(1-q)q^2\sum_{\sigma\in\I_{2n}(123)}q^{2\cdot\des(\sigma)},
}$
\item ${\displaystyle
\sum_{\sigma\in\I_{2n+1}(123)} (-1)^{\maj(\sigma)}q^{\des(\sigma)}=(-1)^nq^2\sum_{\sigma\in\I_{n}(123)}q^{2\cdot\des(\sigma)}.
}$
\end{enumerate}
\end{con}

\smallskip
\subsection{Our work}  For a permutation $\sigma=\sigma_1\cdots\sigma_n\in\mathfrak{S}_n$, let $\lead(\sigma)$ denote the first element of $\sigma$, i.e., $\lead(\sigma)=\sigma_1$.  Recall that the $q$-binomial coefficients are polynomials defined as
\[
{{n}\brack{k}}_q:=\frac{[n]!_q}{[k]!_q[n-k]!_q},
\]
where
$[n]!_q=[1]_q[2]_q\cdots[n]_q$ and $[i]_q=1+q+\cdots+q^{i-1}$.

In addition to answering the above question, one of the main results in this paper is the following enumeration of joint distributions for two statistics of 321-avoiding involutions.

\begin{thm} \label{thm:joint-distribution} We have
\begin{enumerate}
\item ${\displaystyle
\sum_{{\sigma\in\I_{n}(321)}\atop{\des(\sigma)=k}} q^{\maj(\sigma)} =q^{k^2}{{\lceil\frac{n}{2}\rceil}\brack{k}}_q {{\lfloor\frac{n}{2}\rfloor}\brack{k}}_q
}$
\item ${\displaystyle
\sum_{{\sigma\in\I_{n}(321)}\atop{\lead(\sigma)=\ell}} q^{\maj(\sigma)} =\sum_{k\ge 0} q^{k^2+k\ell+\ell-1}{{\lceil\frac{n}{2}\rceil-1}\brack{k}}_q {{\lfloor\frac{n}{2}\rfloor-\ell+1}\brack{k}}_q.
}$
\end{enumerate}
\end{thm}
The proof of the above theorem involves a specialization of the elementary symmetric functions.
We remark that the identity in (i) of Theorem \ref{thm:joint-distribution} has been proved by Barnsbei et al. \cite{BBES} by a different method. 

The second main result is the following refined major-balance identities on 321-avoiding involutions, respecting the leading element. A curious point is that the right hand side of the  identity in (iv) of Theorem \ref{thm:ARM-identity-for-lead-InA321} is a combination of the generating functions of leading element for $\I_{2n+1}(321)$ and $\I_{2n}(321)$.

\medskip
\begin{thm} \label{thm:ARM-identity-for-lead-InA321} For all $n\ge 1$, we have
\begin{enumerate}
\item ${\displaystyle
\sum_{\sigma\in\I_{4n}(321)} (-1)^{\maj(\sigma)}q^{\lead(\sigma)} =\frac{1}{q}\sum_{\sigma\in\I_{2n}(321)}q^{2\cdot\lead(\sigma)},
}$
\item ${\displaystyle
\sum_{\sigma\in\I_{4n+2}(321)} (-1)^{\maj(\sigma)}q^{\lead(\sigma)} =\left(\frac{1}{q}-1\right)\sum_{\sigma\in\I_{2n+1}(321)}q^{2\cdot\lead(\sigma)},
}$
\item ${\displaystyle
\sum_{\sigma\in\I_{4n+3}(321)} (-1)^{\maj(\sigma)}q^{\lead(\sigma)}=\left(\frac{2}{q}-1\right)\sum_{\sigma\in\I_{2n+1}(321)}q^{2\cdot\lead(\sigma)}.
}$
\item ${\displaystyle
\sum_{\sigma\in\I_{4n+1}(321)} (-1)^{\maj(\sigma)}q^{\lead(\sigma)}=(\frac{1}{q}-1)\sum_{\sigma\in\I_{2n+1}(321)}q^{2\cdot\lead(\sigma)}+\sum_{\sigma\in\I_{2n}(321)}q^{2\cdot\lead(\sigma)}.
}$
\end{enumerate}
\end{thm}

\section{A bijection between 321-avoiding permutations and grand Dyck paths}  Let $m,n$ be positive integers. A \emph{Dyck path} of length $2n$ is a lattice path from $(0,0)$ to $(n,n)$, using {\em north} step $(0,1)$ and {\em east} step $(1,0)$, that stays weakly above the line $y=x$. A \emph{partial Dyck path} of length $n$ is a lattice path from $(0,0)$ to the line $x=n$ staying weakly above the line $y=x$. Let $\P_n$ be the set of partial Dyck paths of length $n$.
A \emph{grand Dyck path} of length $m+n$ is a lattice path from $(0,0)$ to $(m,n)$  without restriction.
Let $\B(n,m)$ denote the set of grand Dyck paths from $(0,0)$ to $(m,n)$. Let $\N$ and $\E$ denote a north
step and an east step, respectively.

We shall give combinatorial proofs of Theorem \ref{thm:joint-distribution} and Theorem \ref{thm:ARM-identity-for-lead-InA321}  on the basis of a bijection between $\I_n(321)$ and $\B(\lfloor\frac{n}{2}\rfloor,\lceil\frac{n}{2}\rceil)$ given by Barnabei et al. \cite{BBES}.
With the partial Dyck paths $\P_n$ being the intermediate stage, the bijection is the composition of two maps $\delta:\I_n(321)\rightarrow\P_n$ and $\xi:\P_n\rightarrow\B(\lfloor\frac{n}{2}\rfloor,\lceil\frac{n}{2}\rceil)$. 
First, we describe the map $\delta:\I_n(321)\rightarrow\P_n$ given by Deutsch et al. \cite{DRS}.

\subsection{The bijection $\delta:\I_n(321)\rightarrow\P_n$.} Given a permutation $\sigma=\sigma_1\cdots\sigma_n\in \I_n(321)$, we associate $\sigma$ with a path $\tau=\delta(\sigma)=z_1\cdots z_n\in\P_n$, where  $z_i=\N$ if $\sigma_i\ge i$ and $z_i=\E$ if $\sigma_i<i$.

To find $\delta^{-1}$, we label the steps of $\tau$ from left to right by $1,2,\dots,n$. Proceeding from right to left across $\tau$, couple each $\E$ step with the nearest uncoupled $\N$ step to its left. Then the cycle structure of the involution $\delta^{-1}(\tau)$ can be determined by taking the labels of a coupled pair as a transposition and an uncoupled $\N$ step as a fixed point. Recall that a permutation is an involution if its cycle structure contains no cycle of length greater than two.

\begin{exa} \label{exa:partial-Dyck} {\rm Consider $\sigma=2\, 1\, 3\, 6\, 7\, 4\, 5\, 8\, 10\, 9\, 11\in\I_{11}(321)$. The partial Dyck path $\tau=\delta(\sigma)$ is shown on the left hand side of Figure \ref{fig:lead-grand-Dyck-path}. For the inverse map,  if we label the steps from left to right by $1,2,\dots,11$ and traverse $\tau$ backward, then we obtain the cycle structure of $\sigma=\delta^{-1}(\tau)$, i.e., $(1\, 2)(3)(4\,6)(5\,7)(8)(9\,10)(11)$.
}
\end{exa}

\begin{figure}[ht]
\begin{center}
\psfrag{xi}[][b][1]{$\xi$}
\includegraphics[width=3in]{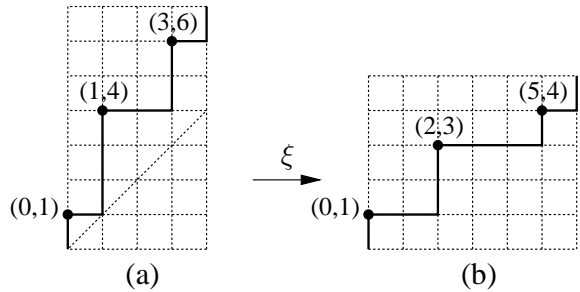}
\end{center}
\caption{\small An example for the two-stage bijection $\xi\circ\delta$ between $\P_{11}$ and $\B(5,6)$.}
\label{fig:lead-grand-Dyck-path}
\end{figure}

A \emph{peak} at position $i$ of $\tau$ is an occurrence $z_iz_{i+1}=\N\E$, which is sometimes identified with the point $p$ between $z_i$ and $z_{i+1}$. Note that the coordinate $(x,y)$ of $p$ satisfies $x+y=i$ and that every descent $\sigma_i>\sigma_{i+1}$ of $\sigma$ is carried to a peak at position $i$ of $\tau$. We observe that if all fixed points in $\sigma$ are ignored, each east step of $\tau$ must be coupled with a remaining north step to its left.
A \emph{valley} of $\tau$ is an occurrence of $\E\N$.
A lattice point with coordinate $(x,y)$ is said to be \emph{even} (\emph{odd}, respectively) if $x+y$ is even (odd, respectively). For convenience, we say that a peak or valley $p$ is  \emph{odd} (\emph{even}, respectively) if $p$ is an odd  (even, respectively) lattice point.
Let $\sump(\tau)$ be the sum of the $x$-coordinates and $y$-coordinates of all peaks in $\tau$.

\medskip
\noindent
\subsection{The bijection $\xi:\P_{n}\rightarrow\B(\lfloor\frac{n}{2}\rfloor,\lceil\frac{n}{2}\rceil)$.}
Let $\tau\in\P_{n}$ be a partial Dyck path with $k$ more $\N$ steps than $\E$ steps ($0\le k\le n$). Match the $\N$ steps and $\E$ steps that face each other, in the sense that the line segment from the midpoint of $\N$ to the midpoint of $\E$ has slope 1 and stays below the path $\tau$.
Then we construct the path $\xi(\tau)\in\B(\lfloor\frac{n}{2}\rfloor,\lceil\frac{n}{2}\rceil)$ by changing the first $\lceil\frac{k}{2}\rceil$ unmatched $\N$ steps into $\E$ steps.  Note that a peak $(x,y)$ in $\tau$ is carried to a peak $(x',y')$ in $\xi(\tau)$ with  $x+y=x'+y'$.

\begin{exa} {\rm
Following Example \ref{exa:partial-Dyck}, consider the path $\tau=z_1\cdots z_{11}\in\P_{11}$ shown on the left hand side of Figure \ref{fig:lead-grand-Dyck-path}. Note that $\sump(\tau)=15$. The unmatched $\N$ steps are $z_3,z_8$ and $z_{11}$. Then the corresponding path $\xi(\pi)\in\B(5,6)$ is obtained from $\tau$ by changing $z_3$ and $z_8$  into $\E$ steps, shown on the right hand side of Figure \ref{fig:lead-grand-Dyck-path}.  Note that the peaks $(1,4)$ and $(3,6)$ of $\tau$ are carried to the peaks $(2,3)$ and $(5,4)$ of $\xi(\tau)$, respectively and $\sump(\xi(\tau))=15$.
}
\end{exa}

To construct $\xi^{-1}$, given a grand Dyck path $\pi\in\B(\lfloor\frac{n}{2}\rfloor,\lceil\frac{n}{2}\rceil)$, match the $\N$ steps and $\E$ steps that face each other in $\pi$. Then the path $\xi^{-1}(\pi)$ is recovered from $\pi$ by changing the remaining unmatched $\E$ steps into $\N$ steps.

\medskip
The following properties about the statistics $\maj(\sigma)$ and $\lead(\sigma)$ hold.

\begin{lem} \label{lem:lead=ell} Given a  permutation $\sigma\in \I_{n}(321)$ with $\lead(\sigma)=\ell$, let $\tau=\delta(\sigma)\in\P_{n}$ and let $\pi=\xi(\tau)\in\B(\lfloor\frac{n}{2}\rfloor,\lceil\frac{n}{2}\rceil)$.  Then the following results hold.
\begin{enumerate}
\item $\maj(\sigma)= \sump(\tau)=\sump(\pi)$.
\item $1\le \ell \le \lfloor\frac{n}{2}\rfloor+1$.
\item The path $\pi$ passes through the points $(0,\ell-1)$ and $(1,\ell-1)$.
\item The number of permutations $\sigma\in \I_{n}(321)$ with $\lead(\sigma)=\ell$ is ${{n-\ell}\choose{\lceil\frac{n}{2}\rceil-1}}$.
\end{enumerate}
\end{lem}

\begin{proof}  Note that every $i\in\Des(\sigma)$ corresponds to a peak $(x,y)$ of $\tau$ and a peak $(x',y')$ of $\pi$ with $x+y=x'+y'=i$. The assertion (i) follows.

Since $\lead(\sigma)=\ell$, $\sigma_{\ell}=1$. We observe that either $\ell=1$ or $\sigma_1<\cdots<\sigma_{\ell-1}$ if $\ell>1$ since $\sigma$ is 321-avoiding. If $\ell>1$ then the entries $2,\dots,\ell-1$ appear to the right of $\sigma_{\ell}$ and hence $2\ell-2\le n$. Moreover, by the construction of the maps $\delta$ and $\xi$, we observe that the grand Dyck path $\pi$ has the prefix $\N^{\ell-1}\E$. The assertions (ii) and (iii) follow.

Note that the number of permutation $\sigma\in \I_{n}(321)$ with $\lead(\sigma)=\ell$ coincides with the number of lattice paths from the point $(1,\ell-1)$ to the point $(\lceil\frac{n}{2}\rceil,\lfloor\frac{n}{2}\rfloor)$. The assertion (iv) follows.
\end{proof}


\section{A combinatorial proof of Theorem \ref{thm:joint-distribution}}

Define the $k$-th \emph{elementary symmetric functions} in $n$ variables
\[
e_k=e_k(x_1,x_2,\dots,x_n) := \sum_{1\le i_1<i_2<\cdots<i_k\le n} x_{i_1}x_{i_2}\cdots x_{i_k}.
\]
Note that $e_0(x_1,\dots,x_n)=1$ and $e_k(x_1,\dots,x_n)=0$ for $k>n$.
Recall the principle specialization
of $e_k(x_1,\dots,x_n)$ (see e.g. \cite[Prop. 7.8.3]{EC2}):
\[
e_k(1,1,\dots,1)={{n}\choose{k}}\quad\mbox{ and}\quad e_k(1,q,\dots,q^{n-1})=q^{{{k}\choose{2}}}{{n}\brack{k}}_q.
\]

\smallskip
\noindent
\emph{Proof of Theorem \ref{thm:joint-distribution}.} (i) By the bijection $\xi\circ\delta$ between $\I_n(321)$ and $\B(\lfloor\frac{n}{2}\rfloor,\lceil\frac{n}{2}\rceil)$, a permutation $\sigma\in\I_n(321)$ with $\des(\sigma)=k$ is mapped to a grand Dyck path $\pi=\xi(\delta(\sigma))$ with $k$ peaks, say $(x_1,y_1),\dots,(x_k,y_k)$ with $0\le x_1<\cdots <x_k\le\lceil\frac{n}{2}\rceil-1$ and $1\le y_1<\cdots<y_k\le\lfloor\frac{n}{2}\rfloor$. Moreover, $\maj(\sigma)=\sump(\pi)=x_1+\cdots+x_k+y_1+\cdots+y_k$. Hence
\begin{align*}
\sum_{{\sigma\in\I_n(321)}\atop{\des(\sigma)=k}} q^{\maj(\sigma)}
    &=\sum_{{0\le x_1<\cdots<x_k\le\lceil\frac{n}{2}\rceil-1}\atop{1\le y_1<\cdots<y_k\le\lfloor\frac{n}{2}\rfloor}} q^{x_1+\cdots+x_k}\cdot q^{y_1+\cdots+y_k} \\
    &= e_k(1,q,\dots,q^{\lceil\frac{n}{2}\rceil-1})\cdot e_k(q,q^2,\dots,q^{\lceil\frac{n}{2}\rceil})\\
    &=q^{k^2}{{\lceil\frac{n}{2}\rceil}\brack{k}}_q {{\lfloor\frac{n}{2}\rfloor}\brack{k}}_q.
\end{align*}

(ii) Let $\sigma\in\I_n(321)$ be a permutation with $\lead(\sigma)=\ell$ and $\des(\sigma)=k$. Then the corresponding grand Dyck path $\pi=\xi(\delta(\sigma))$ can be factorized as $\pi=\N^{\ell-1}\E\mu$. If $\ell=1$ then the segment $\mu$ contains $k$ peaks, say $(x_1,y_1),\cdots,(x_k,y_k)$ with $1\le x_1<\cdots<x_k\le\lceil\frac{n}{2}\rceil-1$ and $1\le y_1<\cdots<y_k\le\lfloor\frac{n}{2}\rfloor$. Hence
\begin{align*}
\sum_{{{\sigma\in\I_n(321)}\atop{\des(\sigma)=k}}\atop{\lead(\sigma)=1}} q^{\maj(\sigma)}
    &= e_k(q,q^2,\dots,q^{\lceil\frac{n}{2}\rceil-1})\cdot e_k(q,q^2,\dots,q^{\lceil\frac{n}{2}\rceil})\\
    &=q^{k^2+k}{{\lceil\frac{n}{2}\rceil-1}\brack{k}}_q {{\lfloor\frac{n}{2}\rfloor}\brack{k}}_q.
\end{align*}
Otherwise $\ell>1$ , and the segment $\mu$ contains another $k-1$ peaks, say $(x_1,y_1),\dots,(x_{k-1},y_{k-1})$ with $1\le x_1<\cdots<x_{k-1}\le\lceil\frac{n}{2}\rceil-1$ and $\ell\le y_1<\cdots<y_{k-1}\le\lfloor\frac{n}{2}\rfloor$. Hence
\begin{align*}
\sum_{{{\sigma\in\I_n(321)}\atop{\des(\sigma)=k}}\atop{\lead(\sigma)=\ell}} q^{\maj(\sigma)}
    &= q^{\ell-1}e_{k-1}(q,q^2,\dots,q^{\lceil\frac{n}{2}\rceil-1})e_{k-1}(q^{\ell},q^{\ell+1},\dots,q^{\lfloor\frac{n}{2}\rfloor})\\
    &=q^{(k-1)^2+(k-1)\ell+\ell-1}{{\lceil\frac{n}{2}\rceil-1}\brack{k-1}}_q {{\lfloor\frac{n}{2}\rfloor-\ell+1}\brack{k-1}}_q.
\end{align*}
The assertion follows. \qed

\medskip
We remark that Barnabei et al. \cite{BBES} proved (i) of Theorem \ref{thm:joint-distribution} by establishing a bijection between the paths in $\B(\lfloor\frac{n}{2}\rfloor,\lceil\frac{n}{2}\rceil)$ and the partitions whose Young diagrams fit inside the $\lfloor\frac{n}{2}\rfloor\times\lceil\frac{n}{2}\rceil$-rectangle so that the descent set of $\sigma\in\I_n(321)$ is carried to the hook-decomposition of the mapped partition.

\medskip
With the result in (ii) of Theorem \ref{thm:joint-distribution}, we give an arithmetic verification of Theorem \ref{thm:ARM-identity-for-lead-InA321} as follows.
For positive integers $m, n$, we have the following facts (i) $[m]_{q=-1}=0$ if and only if $m$ is even, and  (ii) if $m,n$ have the same parity, then
\[
  \lim_{q\rightarrow -1} \frac{[n]_q}{[m]_q}=
 \left\{
\begin{array}{cl} \frac{n}{m} & \mbox{\rm if $m, n$ are even}\\
1 & \mbox{\rm if $m,n$ are odd.}
       \end{array}
\right.
\]
Making use of the above facts, we observe that
\begin{equation} \label{eqn:q-binomial}
{{n}\brack {k}}_{q=-1}=\lim_{q\rightarrow -1}\frac{[n]_q[n-1]_q\cdots[n-k+1]_q}{[1]_q[2]_q\cdots[k]_q}
=\left\{ \begin{array}{ll}
              0 & \text{if $n$ is even and $k$ is odd}\\
              {\displaystyle{ {{\lfloor\frac{n}{2}\rfloor}\choose{\lfloor\frac{k}{2}\rfloor}} }} & \text{otherwise.}
             \end{array}
     \right.
\end{equation}

Now, we verify the identity in (i) of Theorem \ref{thm:ARM-identity-for-lead-InA321}.
By (ii) of Theorem \ref{thm:joint-distribution}, we have the left hand side
\[
\sum_{{\sigma\in\I_{4n}(321)}\atop{\lead(\sigma)=\ell}} (-1)^{\maj(\sigma)}
  = \lim_{q\rightarrow -1} \sum_{k\ge 0} q^{k^2+k\ell+\ell-1}{{2n-1}\brack{k}}_q {{2n-\ell+1}\brack{k}}_q. \\
\]
If $\ell$ is odd, say $\ell=2\ell'-1$ then by Eq.\,(\ref{eqn:q-binomial}) we have
\begin{align*}
\sum_{{\sigma\in\I_{4n}(321)}\atop{\lead(\sigma)=2\ell'-1}} (-1)^{\maj(\sigma)} &= \sum_{k'\ge 0} {{n-1}\choose{k'}}{{n-\ell'+1}\choose{k'}} \\
&={{2n-\ell'}\choose{n-1}}=|\{\sigma\in\I_{2n}(321): \lead(\sigma)=\ell'\}|.
\end{align*}
If $\ell$ is even, say $\ell=2\ell'$, then
\[
\sum_{{\sigma\in\I_{4n}(321)}\atop{\lead(\sigma)=2\ell'}} (-1)^{\maj(\sigma)}=\sum_{k\ge 0} {{n-1}\choose{\lfloor\frac{k}{2}\rfloor}}{{n-\ell'}\choose{\lfloor\frac{k}{2}\rfloor}}=0.
\]
This agrees with the right hand side of (i) of Theorem \ref{thm:ARM-identity-for-lead-InA321}. The other identities (ii), (iii) and (iv) of Theorem \ref{thm:ARM-identity-for-lead-InA321} can be verified in a similar manner. \qed

\medskip
\section{A combinatorial proof of Theorem \ref{thm:ARM-identity-for-lead-InA321}}
For any grand Dyck path $\pi\in\B(n,m)$, we factorize $\pi$ as $\pi=\mu_0\mu_1\cdots\mu_d$, where each segment $\mu_{2i}$ ($\mu_{2i+1}$, respectively) is a maximal sequence of consecutive $\N$ steps ($\E$ steps, respectively). This is called the \emph{primal factorization} of $\pi$. Note that $\mu_0$ is empty if $\pi$ starts with an east step. 

According to the length of $\mu_0$, we partition the set $\B(n,m)$ of into subsets $\B_j(n,m)$ for $0\le j\le n$, where
$\B_j(n,m)$ consists of the paths pass the points $(0,j)$ and $(1,j)$. By (iii) of Lemma \ref{lem:lead=ell}, we have the following result.

\begin{lem} For $0\le j\le \lfloor\frac{n}{2}\rfloor$,
the paths in $\B_j(\lfloor\frac{n}{2}\rfloor,\lceil\frac{n}{2}\rceil)$ are in one-to-one correspondence with the permutations $\sigma\in\I_n(321)$ with $\lead(\sigma)=j+1$.
\end{lem}

\medskip
\subsection{The case $\Phi_1:\B(2n,2n)\rightarrow\B(2n,2n)$} To prove (i) of Theorem \ref{thm:ARM-identity-for-lead-InA321}, we shall establish a $\sump$-parity-reversing involution $\Phi_1:\B(2n,2n)\rightarrow\B(2n,2n)$ while preserving the initial segment from the beginning to the first east step. Let $\F(2n,2n)\subseteq\B(2n,2n)$ be the set of fixed points of the map $\Phi_1$. The set $\F(2n,2n)$ can be constructed from $\B(n,n)$ as follows.

For each path $\omega\in\B(n,n)$, we form a path $\gamma(\omega)$ by duplicating every step of $\omega$. Then $\gamma(\omega)\in\B(2n,2n)$.
Note that the peaks and valleys of $\gamma(\omega)$ are all even lattice points. Moreover, every path without odd peaks and odd valleys in $\B(2n,2n)$ can be reduced to a path in $\B(n,n)$ by a reverse operation.
For example,
for $\omega=\N\E\E\N\E\N\in\B(3,3)$, the path $\gamma(\omega)$ is shown as Figure \ref{fig:lead-duplicate-1}.

\begin{figure}[ht]
\begin{center}
\includegraphics[width=1.2in]{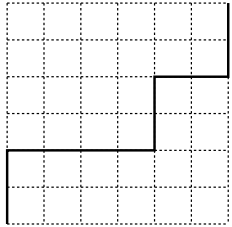}
\end{center}
\caption{\small Construction of the path $\gamma(\omega)$ for $\omega=\N\E\E\N\E\N\in\B(3,3)$.}
\label{fig:lead-duplicate-1}
\end{figure}

The set $\F(2n,2n)$ is defined by
\[
\F(2n,2n)=\{\gamma(\omega): \omega\in\B(n,n)\},
\]
and for $0\le i\le 2n$, the subset $\F_i(2n,2n)$ is defined by
\[
\F_i(2n,2n)=\F(2n,2n)\cap\B_i(2n,2n).
\]
Note that $\F_i(2n,2n)$ is empty if $i$ is odd. We have following immediate observation.

\medskip
\begin{lem} \label{lem:all-even} For $0\le j\le n$ and any path $\omega\in\B_j(n,n)$, the following properties hold.
\begin{enumerate}
\item $\gamma(\omega)\in\F_{2j}(2n,2n)$ and $\sump(\gamma(\omega))=2\cdot\sump(\omega)$.
\item $|\F_{2j}(2n,2n)|=|\B_j(n,n)|$ and $|\F_{2j+1}(2n,2n)|=0$.
\item The set $\F(2n,2n)$ consists of all the paths without odd peaks and odd valleys in $\B(2n,2n)$.
\end{enumerate}
\end{lem}

Now, we construct the involution $\Phi_1$ on $\B(2n,2n)$.

\smallskip
\noindent
{\bf Algorithm A}

Given a path $\pi\in\B(2n,2n)$, let $\pi=\mu_0\mu_1\cdots\mu_d$ be the primal factorization of $\pi$. If every segment $\mu_i$ contains an even number of steps then $\Phi_1(\pi)=\pi$. Otherwise, find the greatest integer $k$ such that $\mu_k$ contains an odd number of steps. The path $\Phi_1(\pi)$ is obtained from $\pi$ by interchanging the first step of $\mu_k$ and the last step of $\mu_{k-1}$.

\begin{exa} \label{exa:case4n} {\rm Let $\pi$ be the path shown on the left hand side of Figure \ref{fig:sump-reverse-1}, with the primal factorization $\pi=\mu_0\mu_1\cdots\mu_5$. Then $\mu_4=\N\N\N$ is the last segment of odd length. Hence $\Phi_1(\pi)$ is obtained from $\pi$ by interchanging the first step of $\mu_4$ and the last step of $\mu_3$, shown on the right hand side of Figure \ref{fig:sump-reverse-1}.
}
\end{exa}

\begin{figure}[ht]
\begin{center}
\psfrag{Phi-1}[][b][1]{$\Phi_1$}
\includegraphics[width=3.25in]{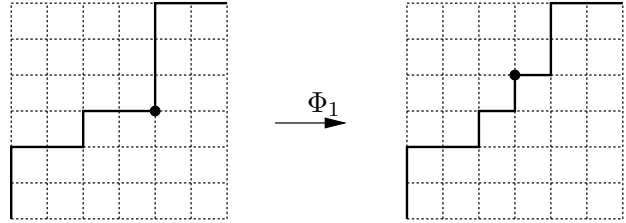}
\end{center}
\caption{\small An example for $\sump$-parity-reversing involution on grand Dyck paths.}
\label{fig:sump-reverse-1}
\end{figure}

\smallskip
\begin{lem} \label{lem:preserve-mu0} For the primal factorization $\pi=\mu_0\mu_1\cdots\mu_d$ of a path $\pi\in\B(2n,2n)$, if $k$ is the greatest integer such that $\mu_k$ contains an odd number of steps then $k\neq 0$ and $k\neq 1$, i.e., $\Phi_1(\pi)$ preserves the first segment $\mu_0$ of $\pi$.
\end{lem}

\begin{proof}
The assertion follows from the fact that $\pi$ has $2n$ east steps and $2n$ north steps.
\end{proof}

\smallskip
\begin{pro} \label{pro:parity-reversing-4n} For $0\le i\le 2n$, the map $\Phi_1$ establishes a refined involution on $\B_i(2n,2n)-\F_i(2n,2n)$. Moreover, a path $\pi$ is carried to a path $\Phi_1(\pi)$ such that $\sump(\Phi_1(\pi))$ has the opposite parity of $\sump(\pi)$.
\end{pro}

\begin{proof} Given a path $\pi\in\B(2n,2n)$, suppose $\Phi_1(\pi)\neq \pi$. By Lemma \ref{lem:all-even}, $\pi$ contains a peak (or valley) which is an odd lattice point. By algorithm A, we find the last odd peak (or valley) $p$. The path $\Phi(\pi)$ is obtained by interchanging the $\N$ and $\E$ steps adjacent at $p$. This changes the last odd peak (or valley) of $\pi$ into the last odd valley (or peak) of $\Phi(\pi)$. Moreover, by Lemma \ref{lem:preserve-mu0},  $\Phi_1$ is an involution, restricted to each subset $\B_i(2n,2n)-\F_i(2n,2n)$. We observe that there is exactly one odd lattice point affected. Hence $\sump(\Phi_1(\pi))$ has the opposite parity of $\sump(\pi)$.
\end{proof}

\noindent
\emph{Proof of (i) of Theorem \ref{thm:ARM-identity-for-lead-InA321}.}
\begin{align*}
\sum_{\sigma\in\I_{4n}(321)} (-1)^{\maj(\sigma)} q^{\lead(\sigma)} &= \sum_{i=0}^{2n} \left( \sum_{\pi\in\B_i(2n,2n)} (-1)^{\sump(\sigma)} \right) q^{i+1} \\
  &= \sum_{j=0}^{n} \left(\sum_{\pi\in\F_{2j}(2n,2n)}q^{2j+1} \right)\\
  &= \sum_{j=0}^{n} |\B_j(n,n)|q^{2j+1} \\
  &= \frac{1}{q}\sum_{\sigma\in\I_{2n}(321)} q^{2\cdot\lead(\sigma)}.
\end{align*}

\subsection{The case $\Phi_2:\B(2n+1,2n+1)\rightarrow\B(2n+1,2n+1)$} To prove (ii) of Theorem \ref{thm:ARM-identity-for-lead-InA321}, we shall establish a $\sump$-parity-reversing involution  $\Phi_2:\B_i(2n+1,2n+1)\rightarrow\B_i(2n+1,2n+1)$ for $0\le i\le 2n+1$. Let $\F(2n+1,2n+1)\subseteq\B(2n+1,2n+1)$ be the set of fixed points of the map $\Phi_2$.  For $0\le i\le 2n+1$, we define
\[
\F_i(2n+1,2n+1)=\F(2n+1,2n+1)\cap\B_i(2n+1,2n+1).
\]
The set $\F(2n+1,2n+1)$ can be constructed from $\B(n,n+1)$ as follows.
For each path $\omega\in\B_j(n,n+1)$ ($0\le j\le n$), we form a path $\gamma(\omega)$ by duplicating every step of $\omega$. Note that $\gamma(\omega)$ is from $(0,0)$ to $(2n+2,2n)$ with the prefix $\N^{2j}\E\E$. Factorize $\gamma(\omega)$ as $\gamma(\omega)=\N^{2j}\E\E\beta$. Then we create two paths $\phi_1(\omega),\phi_2(\omega)\in\B(2n+1,2n+1)$ from $\gamma(\omega)$ by
\begin{equation*} \label{eqn:phi}
\phi_1(\omega) =\N^{2j}\E\N\beta, \qquad
\phi_2(\omega) =\N^{2j+1}\E\beta,
\end{equation*}
i.e., $\phi_1(\omega)$ (respectively, $\phi_2(\omega)$) is obtained from $\gamma(\omega)$ by changing the second (respectively, first) east step into a north step. For example,
let $\omega=\N\E\E\N\E\in\B_1(2,3)$. Then the path $\gamma(\omega)$ is shown on the left hand side of Figure \ref{fig:lead-duplicate-2} and the paths $\phi_1(\omega), \phi_2(\omega)$ are shown on the right hand side of Figure \ref{fig:lead-duplicate-2}.

\begin{figure}[ht]
\begin{center}
\psfrag{gamma}[][b][1]{$\gamma(\omega)$}
\psfrag{phi-1}[][b][1]{$\phi_1(\omega)$}
\psfrag{phi-2}[][b][1]{$\phi_2(\omega)$}
\includegraphics[width=3.75in]{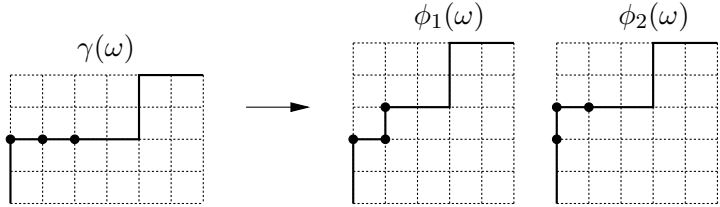}
\end{center}
\caption{\small Construction of the paths $\phi_1(\omega)$ and $\phi_2(\omega)$ for $\omega=\N\E\E\N\E$.}
\label{fig:lead-duplicate-2}
\end{figure}

For $0\le j\le n$, define
\begin{align*}
\F_{2j}(2n+1,2n+1) &=\{\phi_1(\omega): \omega\in\B_j(n,n+1)\}, \\
\F_{2j+1}(2n+1,2n+1) &=\{\phi_2(\omega): \omega\in\B_j(n,n+1)\}.
\end{align*}

Note that the segment $\beta$ of $\phi_1(\omega), \phi_2(\omega)$ goes from $(1,2\ell+1)$ to $(2n+1,2n+1)$ and that every segment in the primal factorization of $\beta$ contains an even number of steps.
We have following immediate observation.
\smallskip
\begin{lem} \label{lem:unique-odd} For $0\le j\le n$ and any path $\omega\in\B_j(n,n+1)$,  the following properties hold.
\begin{enumerate}
\item The path $\phi_1(\omega)\in\B_{2j}(2n+1,2n+1)$ contains a unique odd valley $(1,2j)$ and no odd peaks.
\item The path $\phi_2(\omega)\in\B_{2j+1}(2n+1,2n+1)$ contains a unique odd peak $(0,2j+1)$ and no odd valleys.
\item $|\F_{2j}(2n+1,2n+1)|=|\F_{2j+1}(2n+1,2n+1)|=|\B_j(n,n+1)|$.
\end{enumerate}
\end{lem}
It follows that $\sump(\phi_1(\omega))$ is even and $\sump(\phi_2(\omega))$ is odd. Note that every path $\pi\in \B(2n+1,2n+1)$ contains at least one odd valley or odd peak since $\pi$ has $2n+1$ east steps and $2n+1$ north steps. In fact,
the set $\F(2n+1,2n+1)$ consists of all the paths in $\B(2n+1,2n+1)$ either containing a unique odd valley in the line $x=1$ and no odd peaks, or containing a unique odd peak in the line $x=0$ and no odd valleys.

Now, we construct the involution $\Phi_2$ on $\B(2n+1,2n+1)$.

\smallskip
\noindent
{\bf Algorithm B}

Given a path $\pi\in\B(2n+1,2n+1)$, let $\mu_0\mu_1\cdots\mu_d$ be the primal factorization of $\pi$. According to the parity of the length of $\mu_0$, there are two cases.

Case 1.  $\mu_0$ is odd, say $\mu_0=\N^{2j+1}$. Find the greatest integer $k\ge 1$ such that $\mu_k$ contains an odd number of steps. If $k=1$ then let $\Phi_2(\pi)=\pi$. Otherwise, the path $\Phi_1(\pi)$ is obtained from $\pi$ by interchanging the first step of $\mu_k$ and the last step of $\mu_{k-1}$.

Case 2. $\mu_0$ is even, say $\mu_0=\N^{2j}$. If $\mu_1=\E$ and $\mu_2$ is the only other segment containing an odd number of steps (i.e., $\mu_t$ is of even length for all $t\ge 3$) then let $\Phi_2(\pi)=\pi$. Otherwise, find the greatest integer $k\ge 3$ such that $\mu_k$ contains an odd number of steps. Then the path $\Phi_2(\pi)$ is obtained from $\pi$ by interchanging the first step of $\mu_k$ and the last step of $\mu_{k-1}$.

It is obvious that the map $\Phi_2$ preserves the segment $\mu_0$, i.e., $\Phi_2$ is a map restricted to each subset $\B_i(2n+1,2n+1)$ for $0\le i\le 2n+1$.

\begin{exa} \label{exa:case4n+2} {\rm Let $\pi$ be the path shown on the left hand side of Figure \ref{fig:sump-reverse-2A}, with the primal factorization $\pi=\mu_0\mu_1\cdots\mu_7$. Then $\mu_5=\E\E\E$ is the last segment of odd length. Hence $\Phi_2(\pi)$ is obtained from $\pi$ by interchanging the first step of $\mu_5$ and the last step of $\mu_4$, shown on the right hand side of Figure \ref{fig:sump-reverse-2A}.
}
\end{exa}

\begin{figure}[ht]
\begin{center}
\psfrag{Phi-2}[][b][1]{$\Phi_2$}
\includegraphics[width=3.25in]{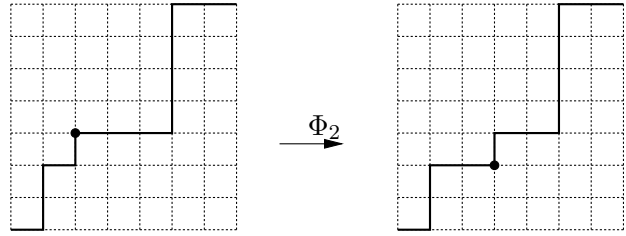}
\end{center}
\caption{\small An example for the map $\Phi_2:\B(7,7)\rightarrow\B(7,7)$.}
\label{fig:sump-reverse-2A}
\end{figure}


By Lemma the same argument as in the proof of Proposition \ref{pro:parity-reversing-4n}, the assertion is proved.

\begin{pro} \label{pro:parity-reversing-4n+2} For $0\le i\le 2n+1$, the map $\Phi_2$ establish a refined involution on $\B_i(2n+1,2n+1)-\F_i(2n+1,2n+1)$. Moreover, a path $\pi$ is carried to a path $\Phi_2(\pi)$ such that $\sump(\Phi_2(\pi))$ has the opposite parity of $\sump(\pi)$.
\end{pro}

\begin{proof}
Given a path $\pi\in\B(2n+1,2n+1)$, suppose $\Phi_2(\pi)\neq \pi$. By Lemma \ref{lem:unique-odd}, $\pi$ either contains an odd peak  $(x,y)$ with $x>0$ or contains an odd valley $(x',y')$ with $x'>1$. By algorithm B, we find the last odd peak (or valley) $p$ and construct the path $\Phi_2(\pi)$ by interchanging the $\N$ and $\E$ steps adjacent at $p$. By the same argument as in the proof of Proposition \ref{pro:parity-reversing-4n}, the assertion is proved.
\end{proof}

\noindent
\emph{Proof of (ii) of Theorem \ref{thm:ARM-identity-for-lead-InA321}.}
\begin{align*}
\sum_{\sigma\in\I_{4n+2}(321)} (-1)^{\maj(\sigma)} q^{\lead(\sigma)} &= \sum_{i=0}^{2n+1} \left( \sum_{\pi\in\B_i(2n+1,2n+1)} (-1)^{\sump(\sigma)} \right)q^{i+1} \\
  &= \sum_{j=0}^{n} \left(\sum_{\pi\in\F_{2j}(2n+1,2n+1)}q^{2j+1}- \sum_{\pi\in\F_{2j+1}(2n+1,2n+1)}q^{2j+2} \right) \\
  &= \sum_{j=0}^{n} |\B_j(n,n+1)|q^{2j+1}- \sum_{j=0}^{n} |\B_j(n,n+1)|q^{2j+2} \\
  &= \left(\frac{1}{q}-1 \right) \sum_{\sigma\in\I_{2n+1}(321)} q^{2\cdot\lead(\sigma)}.
\end{align*}

\smallskip
\subsection{The case $\Phi_3:\B(2n+1,2n+2)\rightarrow\B(2n+1,2n+2)$} To prove (iii) of Theorem \ref{thm:ARM-identity-for-lead-InA321}, we shall establish a $\sump$-parity-reversing involution  $\Phi_3:\B_i(2n+1,2n+2)\rightarrow\B_i(2n+1,2n+2)$ for $0\le i\le 2n+1$. Let $\F(2n+1,2n+2)\subseteq\B(2n+1,2n+2)$ be the set of fixed points of the map $\Phi_3$.  For $0\le i\le 2n+1$, we define
\[
\F_i(2n+1,2n+2)=\F(2n+1,2n+2)\cap\B_i(2n+1,2n+2).
\]

The set $\F(2n+1,2n+2)$ can be constructed from $\B(n,n+1)$ as follows.
For each path $\omega\in\B_j(n,n+1)$  ($0\le j\le n$), we form a path $\gamma(\omega)$ by duplicating every step of $\omega$. Note that $\gamma(\omega)$ is from $(0,0)$ to $(2n+2,2n)$ with the prefix $\N^{2j}\E\E$. Factorize $\gamma(\omega)$ as $\gamma(\omega)=\N^{2j}\E\E\beta$. Then we create three paths $\psi_0(\omega),\psi_1(\omega),\psi_2(\omega)\in\B(2n+1,2n+2)$ from $\gamma(\omega)$ by
\begin{align*}
\psi_0(\omega) &=\gamma(\omega)\N, \\
\psi_1(\omega) &=\N^{2j}\E\N\beta\E, \\
\psi_2(\omega) &=\N^{2j+1}\E\beta\E.
\end{align*}
Note that $\psi_0(\omega)$ is obtained from $\gamma(\omega)$ by appending a north step in the end and that  $\psi_1(\omega)$ (respectively, $\psi_2(\omega)$) is obtained from $\gamma(\omega)$ inserting a north step after (respectively, before) the first east step and moving the second east step to the end. For example,
for $\omega=\N\E\E\N\E\in\B_1(2,3)$, the paths $\psi_0(\omega)$, $\psi_1(\omega)$ and $\psi_2(\omega)$ are shown in Figure \ref{fig:lead-duplicate-3}.

\begin{figure}[ht]
\begin{center}
\psfrag{psi-0}[][b][1]{$\psi_0$}
\psfrag{psi-1}[][b][1]{$\psi_1$}
\psfrag{psi-2}[][b][1]{$\psi_2$}
\includegraphics[width=3.6in]{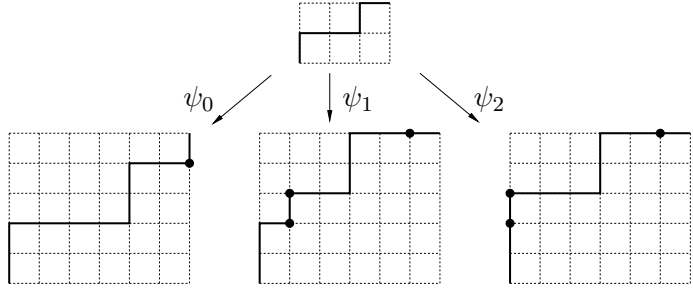}
\end{center}
\caption{\small Construction of the paths $\psi_0(\omega), \psi_1(\omega)$ and $\psi_2(\omega)$ for $\omega=\N\E\E\N\E$.}
\label{fig:lead-duplicate-3}
\end{figure}

\smallskip
We define the refinement of the set $\F(2n+1,2n+2)$. For $0\le j\le n$ let
\begin{align*}
\F_{2j}(2n+1,2n+2) &=\{\psi_0(\omega),\psi_1(\omega): \omega\in\B_j(n,n+1)\}, \\
\F_{2j+1}(2n+1,2n+2) &=\{\psi_2(\omega): \omega\in\B_j(n,n+1)\}.
\end{align*}

We have the following properties of the fixed points $\F(2n+1,2n+2)$.

\smallskip
\begin{lem} For $0\le j\le n$ and any path $\omega\in\B_j(n,n+1)$, the following properties hold.
\begin{enumerate}
\item The path $\psi_0(\omega)$ contains no odd peaks and odd valleys.
\item The path $\psi_1(\omega)$ contains a unique odd valley $(1,2j)$ and no odd peaks.
\item The path $\psi_2(\omega)$ contains a unique odd peak $(0,2j+1)$ and no odd valleys.
\item $|\F_{2j}(2n+1,2n+2)|=2|\B_j(n,n+1)|$ and $|\F_{2j+1}(2n+1,2n+2)|=|\B_j(n,n+1)|$.
\end{enumerate}
\end{lem}
It follows that $\sump(\psi_0(\omega)),\sump(\psi_1(\omega))$ are even and
$\sump(\psi_2(\omega))$ is odd. In fact, $\F(2n+1,2n+2)$ consists of all the paths in $\B(2n+1,2n+2)$ containing on odd valley $(x,y)$ with $x\ge 2$ and no odd peak $(x',y')$ with $x'\ge 1$.

Now, we construct the involution $\Phi_3$ on $\B(2n+1,2n+2)$.

\smallskip
\noindent
{\bf Algorithm C}

Given a path $\pi\in\B(2n+1,2n+2)$, let $z$ denote the last step of $\pi$. Let $\pi'$ be the path obtained from $\pi$ by removing $z$.  We consider the following two cases according to the step $z$.

Case 1.  $z=\N$. Then $\pi'$ goes from $(0,0)$ to $(2n+2,2n)$. Applying algorithm A to the primal factorization of $\pi'$, we determine the path $\Phi_1(\pi')$ associated with $\pi'$ under the map $\Phi_1$.  Then the corresponding path $\Phi_3(\pi)$ is obtained from $\Phi_1(\pi')$ by appending a north step, i.e., $\Phi_3(\pi)=\Phi_1(\pi')\N\in\B(2n+1,2n+2)$.

Case 2. $z=\E$. Then $\pi'$ goes from $(0,0)$ to $(2n+1,2n+1)$. Applying algorithm B to the primal factorization of $\pi'$, we determine the path $\Phi_2(\pi')$ associated with $\pi'$ under the map $\Phi_2$. Then the corresponding path $\Phi_3(\pi)\in\B(2n+1,2n+2)$ is obtained from $\Phi_2(\pi')$ by appending an east step, i.e., $\Phi_3(\pi)=\Phi_2(\pi')\E\in\B(2n+1,2n+2)$.

\smallskip
Note that the construction of the map $\Phi_3$ in Case 1 (respectively, Case 2) of algorithm C is similar to the  construction of $\Phi_1$ by algorithm A (respectively, $\Phi_2$ by algorithm B). The following property of the map $\Phi_3$ can be proved by the same argument as in the proofs of Propositions \ref{pro:parity-reversing-4n} and \ref{pro:parity-reversing-4n+2}.

\begin{pro} \label{pro:parity-reversing-4n+3} For $0\le i\le 2n+1$, the map $\Phi_3$ establishes a refined involution on $\B_i(2n+1,2n+2)-\F_i(2n+1,2n+2)$. Moreover, a path $\pi$ is carried to a path $\Phi_3(\pi)$ such that $\sump(\Phi_3(\pi))$ has the opposite parity of $\sump(\pi)$.
\end{pro}

\noindent
\emph{Proof of (iii) of Theorem \ref{thm:ARM-identity-for-lead-InA321}.}
\begin{align*}
\sum_{\sigma\in\I_{4n+3}(321)} (-1)^{\maj(\sigma)} q^{\lead(\sigma)} &= \sum_{i=0}^{2n+1} \left( \sum_{\pi\in\B_i(2n+1,2n+2)} (-1)^{\sump(\sigma)} \right)q^{i+1} \\
  &=\sum_{j=0}^{n} \left(\sum_{\pi\in\F_{2j}(2n+1,2n+2)}q^{2j+1}- \sum_{\pi\in\F_{2j+1}(2n+1,2n+2)}q^{2j+2} \right) \\
  &= \sum_{j=0}^{n} 2|\B_j(n,n+1)|q^{2j+1}- \sum_{j=0}^{n} |\B_j(n,n+1)|q^{2j+2} \\
  &= \left(\frac{2}{q}-1 \right) \sum_{\sigma\in\I_{2n+1}(321)} q^{2\cdot\lead(\sigma)}.
\end{align*}

\smallskip
\subsection{The case $\Phi_4:\B(2n,2n+1)\rightarrow\B(2n,2n+1)$} To prove (iv) of Theorem \ref{thm:ARM-identity-for-lead-InA321}, we shall establish a $\sump$-parity-reversing involution  $\Phi_4:\B_i(2n,2n+1)\rightarrow\B_i(2n,2n+1)$ for $0\le i\le 2n$. Let $\F(2n,2n+1)\subseteq\B(2n,2n+1)$ be the set of fixed points of the map $\Phi_4$.  For $0\le i\le 2n$, we define
\[
\F_i(2n,2n+1)=\F(2n,2n+1)\cap\B_i(2n,2n+1).
\]
The set $\F(2n,2n+1)$ can be constructed from $\B(n,n+1)$ as follows.
For each path $\omega\in\B_j(n,n+1)$ ($0\le j\le n$), we form a path $\gamma(\omega)$  by duplicating every step of $\omega$. We consider the following two cases according to the last step $z$ of $\omega$:
\begin{itemize}
   \item $z=\E$. Then the last two steps of $\gamma(\omega)$ are east steps. Let $\varphi_0(\omega)$ be the path obtained from $\gamma(\omega)$ by removing the last step.
   \item $z=\N$. Then the last two steps of $\gamma(\omega)$ are north steps. Factorize $\gamma(\omega)$ as $\N^{2j}\E\E\beta\N\N$ and let $\varphi_1(\omega)$ (respectively, $\varphi_2(\omega)$) be the path obtained from $\gamma(\omega)$ by inserting a north step after (respectively, before)  the first east step and then removing the second east step and the last step, i.e.,
\begin{align*}
\varphi_1(\omega) &=\N^{2j}\E\N\beta\N \\
\varphi_2(\omega) &=\N^{2j+1}\E\beta\N
\end{align*}
\end{itemize}
For example, for $\omega_1=\N\E\E\N\E\N\E\in\B_1(3,4)$, the path $\varphi_0(\omega_1)$ is shown as the left hand side of Figure \ref{fig:lead-duplicate-4}. For $\omega_2=\E\E\N\E\N\E\N\in\B_0(3,4)$, the paths $\varphi_1(\omega_2), \varphi_2(\omega_2)$ are shown as the right hand side of Figure \ref{fig:lead-duplicate-4}.

\begin{figure}[ht]
\begin{center}
\psfrag{vphi-0}[][b][1]{$\varphi_0$}
\psfrag{vphi-1}[][b][1]{$\varphi_1$}
\psfrag{vphi-2}[][b][1]{$\varphi_2$}
\includegraphics[width=4.0in]{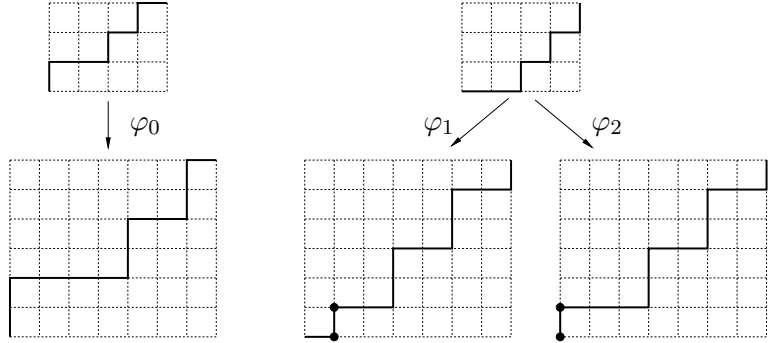}
\end{center}
\caption{\small Examples of the maps $\varphi_0, \varphi_1$ and $\varphi_2$.}
\label{fig:lead-duplicate-4}
\end{figure}

\medskip
We define the refinement of the set $\F(2n,2n+1)$. For $0\le j\le n$, let $\F_{2j}(2n,2n+1)=\F_{2j}^{E}(2n,2n+1)\cup\F_{2j}^{N}(2n,2n+1)$, where
\begin{align*}
\F_{2j}^{E}(2n,2n+1) &=\{\varphi_0(\omega): \omega\in\B_j(n,n+1) \mbox{ ends with an east step}\}, \\
\F_{2j}^{N}(2n,2n+1) &=\{\varphi_1(\omega): \omega\in\B_j(n,n+1) \mbox{ ends with a north step}\}, \\
\F_{2j+1}(2n,2n+1) &=\{\varphi_2(\omega): \omega\in\B_j(n,n+1) \mbox{ ends with a north step}\}.
\end{align*}

We have the following properties of the fixed points $\F(2n+1,2n+2)$.
\smallskip
\begin{lem} For $0\le j\le n$ and any path $\omega\in\B_j(n,n+1)$, the following properties hold.
\begin{enumerate}
\item The path $\varphi_0(\omega)$ contains no odd peaks and odd valleys.
\item The path $\varphi_1(\omega)$ contains a unique odd valley $(1,2j)$ and no odd peaks.
\item The path $\varphi_2(\omega)$ contains a unique odd peak $(0,2j+1)$ and no odd valleys.
\item $|\F_{2j}(2n+1,2n+2)|=|\B_j(n,n+1)|$.
\item $|\F_{2j+1}(2n+1,2n+2)|=|\B_j(n,n+1)|-|\B_j(n,n)|$.
\end{enumerate}
\end{lem}
It follows that $\sump(\varphi_0(\omega)),\sump(\varphi_1(\omega))$ are even and $\sump(\varphi_2(\omega))$ is odd.
Now, we construct the involution $\Phi_4$ on $\B(2n,2n+1)$.

\smallskip
\noindent
{\bf Algorithm D}

Given a path $\pi\in\B(2n,2n+1)$, let $z$ denote the last step of $\pi$. Let $\pi'$ be the path obtained from $\pi$ by removing $z$.  We consider the following two cases according to the step $z$.

Case 1.  $z=\E$. Then $\pi'$ goes from $(0,0)$ to $(2n,2n)$. By algorithm A, we determine the path $\Phi_1(\pi')\in\B(2n,2n)$ associated with $\pi'$ under the map $\Phi_1$. Then the corresponding path $\Phi_4(\pi)$ is obtained from $\Phi_1(\pi')$ by appending an east step in the end, i.e., $\Phi_4(\pi)=\Phi_1(\pi')\E$.

Case 2. $z=\N$. Then $\pi'$ goes from $(0,0)$ to $(2n+1,2n-1)$. By the same method as in algorithm B, we determine the path $\Phi_2(\pi')\in\B(2n-1,2n+1)$ associated with $\pi'$ under the map $\Phi_2$. Then the corresponding path $\Phi_4(\pi)\in\B(2n,2n+1)$ is obtained from $\Phi_2(\pi')$ by appending a north step in the end, i.e., $\Phi_4(\pi)=\Phi_2(\pi')\N$.

The following property of the map $\Phi_4$ can be proved by the same argument as in the proofs of Propositions \ref{pro:parity-reversing-4n} and \ref{pro:parity-reversing-4n+2} since the construction of $\Phi_4$ in Case 1 (respectively, Case 2) of algorithm D is similar to the construction of $\Phi_1$ (respectively, $\Phi_2$).

\begin{pro} \label{pro:parity-reversing-4n+1} For $0\le i\le 2n$, the map $\Phi_4$ establishes a refined involution on $\B_i(2n,2n+)-\F_i(2n,2n+1)$. Moreover, a path $\pi$ is carried to a path $\Phi_4(\pi)$ such that $\sump(\Phi_4(\pi))$ has the opposite parity of $\sump(\pi)$.
\end{pro}

\noindent
\emph{Proof of (iv) of Theorem \ref{thm:ARM-identity-for-lead-InA321}.}
\begin{align*}
\sum_{\sigma\in\I_{4n+1}(321)} (-1)^{\maj(\sigma)} q^{\lead(\sigma)} &= \sum_{i=0}^{2n} \left( \sum_{\pi\in\B_i(2n,2n+1)} (-1)^{\sump(\sigma)} \right)q^{i+1} \\
  &= \sum_{j=0}^{n} \left(\sum_{\pi\in\F_{2j}(2n,2n+1)}q^{2j+1}- \sum_{\pi\in\F_{2j+1}(2n,2n+1)}q^{2j+2} \right) \\
  &= \sum_{j=0}^{n} \big(|\B_j(n,n+1)|- |\B_j(n,n+1)|q+|\B_j(n,n)|q\big)q^{2j+1} \\
  &= \left(\frac{1}{q}-1 \right) \sum_{\sigma\in\I_{2n+1}(321)} q^{2\cdot\lead(\sigma)}+\sum_{\sigma\in\I_{2n}(321)} q^{2\cdot\lead(\sigma)}.
\end{align*}

\section{Analogous results for 123-avoiding involutions}
It is known that by the Robinson-Schensted-Knuth (RSK) algorithm an involution $\sigma\in\I_n(321)$ (respectively, $\sigma\in\I_n(123)$) is associated with a pair $(Q,Q)$ of identical $n$-cell standard Young tableaux with at most two rows (respectively, columns). Let $Q^{\mathtt T}$ be the transpose of $Q$ and let $\sigma^{\mathtt T}$ be the preimage of the pair $(Q^{\mathtt T},Q^{\mathtt T})$ under the RSK correspondence. Then $\sigma\leftrightarrow\sigma^{\mathtt T}$ is a bijection between $\I_n(321)$ and $\I_n(123)$.

\begin{lem} \label{lem: Q-transpose} We have 
\[
\Des(\sigma^{\mathtt T})=\{i: i\not\in \Des(\sigma), 1\le i\le n-1\}.
\]
\end{lem}
\begin{proof}
  It is known that a descent $\sigma_i>\sigma_{i+1}$ in $\sigma$ is translated to the `descent' of the recording tableau $Q$ that the entry $i+1$ is in a row lower than the row of $i$. For $1\le i\le n-1$, if $i\in\Des(\sigma)$ then $i$ (respectively, $i+1$) is in the first (respectively, second) row in $Q$. Then $i+1$ is either in the same column as $i$ or in a column to the left of the column of $i$ in $Q$.  
  Then $i+1$ is not in a lower row than the row of $i$ in $Q^{\mathtt T}$. Hence $i\not\in\Des(\sigma^{\mathtt T})$.
  
  Otherwise, $i\not\in\Des(\sigma)$. Then $i+1$ is not in a row lower than the row of $i$ in $Q$. The element $i$ is in the first row or the second row. In either case, the element $i+1$ is in a column to the right of the column of $i$. Then $i+1$ is in row lower than the row of $i$ in $Q^{\mathtt T}$. Hence $i\in\Des(\sigma^{\mathtt T})$.
\end{proof}

We obtain the joint distribution of major index and descent number for 123-avoiding involutions.

\begin{cor} We have
\[
\sum_{{\sigma\in\I_n(123)}\atop{\des(\sigma)=n-1-k}} q^{\maj(\sigma)}
    =q^{{{n}\choose{2}}+k^2-nk}{{\lceil\frac{n}{2}\rceil}\brack{k}}_q {{\lfloor\frac{n}{2}\rfloor}\brack{k}}_q.
\]
\end{cor}

\begin{proof} Substituting $q^{-1}$ for $q$ in (i) of Theorem \ref{thm:joint-distribution}, we have
\[
\sum_{{\sigma\in\I_n(321)}\atop{\des(\sigma)=k}} q^{-\maj(\sigma)}=q^{-k^2}{{\lceil\frac{n}{2}\rceil}\brack{k}}_{q^{-1}} {{\lfloor\frac{n}{2}\rfloor}\brack{k}}_{q^{-1}}
    =q^{k^2-kn}{{\lceil\frac{n}{2}\rceil}\brack{k}}_q {{\lfloor\frac{n}{2}\rfloor}\brack{k}}_q.
\]
By Lemma \ref{lem: Q-transpose}, $\des(\sigma)=n-1-\des(\sigma^{\mathtt T})$ and $\maj(\sigma)={{n}\choose{2}}-\maj(\sigma^{\mathtt T})$  for any $\sigma\in\I_n(123)$, we have
\begin{align*}
\sum_{{\sigma\in\I_n(123)}\atop{\des(\sigma)=n-1-k}} q^{\maj(\sigma)}
    &= \sum_{{\sigma^{\mathtt T}\in\I_n(321)}\atop{\des(\sigma^{\mathtt T})=k}} q^{{{n}\choose{2}}-\maj(\sigma^{\mathtt T})}\\
    &=q^{{{n}\choose{2}}+k^2-nk}{{\lceil\frac{n}{2}\rceil}\brack{k}}_q {{\lfloor\frac{n}{2}\rfloor}\brack{k}}_q,\end{align*}
as required.
\end{proof}

With the bijection $\sigma\leftrightarrow\sigma^{\mathtt T}$  between $\I_n(321)$ and $\I_n(123)$, we can prove affirmatively Conjecture \ref{con:max-des-identity-for-InA123}. In fact, this result is essentially equivalent to Theorem \ref{thm:ARM-identity-for-InA321}.

\smallskip
\begin{thm} \label{thm:max-des-identity-for-InA123} For all $n\ge 1$, we have
\begin{enumerate}
\item ${\displaystyle
\sum_{\sigma\in\I_{4n}(123)} (-1)^{\maj(\sigma)}q^{\des(\sigma)} =q\sum_{\sigma\in\I_{2n}(123)}q^{2\cdot\des(\sigma)},
}$
\item ${\displaystyle
\sum_{\sigma\in\I_{4n+2}(123)} (-1)^{\maj(\sigma)}q^{\des(\sigma)} =(1-q)q^2\sum_{\sigma\in\I_{2n}(123)}q^{2\cdot\des(\sigma)},
}$
\item ${\displaystyle
\sum_{\sigma\in\I_{2n+1}(123)} (-1)^{\maj(\sigma)}q^{\des(\sigma)}=(-1)^nq^2\sum_{\sigma\in\I_{n}(123)}q^{2\cdot\des(\sigma)}.
}$
\end{enumerate}
\end{thm}

\begin{proof} Note that $\des(\sigma)=n-1-\des(\sigma^{\mathtt T})$ and $\maj(\sigma)={{n}\choose{2}}-\maj(\sigma^{\mathtt T})$  for any $\sigma\in\I_n(123)$.
Making use of the identity in (i) of Theorem \ref{thm:ARM-identity-for-InA321}, we have 
\begin{align*}
\sum_{\sigma\in\I_{4n}(123)} (-1)^{\maj(\sigma)}q^{\des(\sigma)}
    &= \sum_{\sigma^{\mathtt T}\in\I_{4n}(321)} (-1)^{{{4n}\choose{2}}-\maj(\sigma^{\mathtt T})}q^{4n-1-\des(\sigma^{\mathtt T})} \\
    &= q\sum_{\sigma^{\mathtt T}\in\I_{2n}(321)} q^{2(2n-1-\des(\sigma^{\mathtt T}))} \quad\quad\mbox{(by (i) of Theorem \ref{thm:ARM-identity-for-InA321})}\\
    &= q\sum_{\sigma\in\I_{2n}(123)} q^{2\cdot\des(\sigma)}.
\end{align*}
The assertion (i) follows.
Making use of the the identities in (ii) and (iii) of Theorem \ref{thm:ARM-identity-for-InA321}, the assertions (ii) and (iii) can be proved straightforward in the same manner.
\end{proof}

\end{document}